\documentclass{amsart}

\usepackage{amssymb}
\usepackage{mathrsfs}
\usepackage[stretch=10,shrink=10]{microtype}
\usepackage[showframe=false, margin = 1.5 in]{geometry}
\usepackage{mathtools}
\usepackage[shortlabels]{enumitem}
\usepackage{ifthen}
\newcounter{dummy}
\usepackage{xcolor}
\usepackage{stmaryrd}
\usepackage{tikz,tikz-cd}
\usepackage{array}
\usepackage{color}
\usepackage{colortbl}
\usepackage{hyperref}
\hypersetup{
     colorlinks   = true,
     citecolor    = black
}
\usepackage{theoremref}

\makeatletter
\def\thmref@flush{%
   \ifx\thmref@last\empty\else
      \ifthmref@comma, \thmref@finaltrue\fi \thmref@commatrue
      \thmref@last \ifx\thmref@stack\empty\else s\fi \thmref@num 0
      \let\do\thmref@one \thmref@stack
      \ifcase\thmref@num\or\space and\else\thmref@finaltrue, and\fi
      ~\ref{\thmref@head}\let\thmref@stack\empty\fi}
\def\thmref@one#1{\ifnum\thmref@num>0,\fi
   \space\ref{#1}\advance\thmref@num 1\relax}
\makeatother

\makeatletter
\newcommand\myitem[1][]{\item[#1]\refstepcounter{dummy}\def\@currentlabel{#1}}
\makeatother

\newcommand{\E}{\mathbf{E}}

\renewcommand{\P}{\mathbf{P}}

\newcommand{\1}{\mathbf{1}}

\DeclareMathOperator{\Bin}{Bin}

\DeclarePairedDelimiter\abs{\lvert}{\rvert}%
\DeclarePairedDelimiter\floor{\lfloor}{\rfloor}%
\DeclarePairedDelimiter\ceil{\lceil}{\rceil}%
\DeclarePairedDelimiter\ii{\llbracket}{\rrbracket}%

\newcommand{\erase}{\mathsf{GreedyPath}}

\newcommand{\ZZ}{\mathbb{Z}}

\newcommand{\NN}{\mathbb{N}}

\newcommand{\instr}{\mathsf{instr}}
\newcommand{\Instr}{\mathsf{Instr}}
\newcommand{\eosos}[1]{\mathcal{O}_{#1}}
\newcommand{\ip}[1]{\mathcal{I}_{#1}}
\newcommand{\s}{\mathfrak{s}}

\newcommand{\critical}{\rho}
\newcommand{\lpcritical}{\rho_*}

\newcommand{\rt}[2]{\mathcal{R}_{#2}(#1)}
\newcommand{\lt}[2]{\mathcal{L}_{#2}(#1)}
\newcommand{\emax}{e_{\max}}

\DeclareFontEncoding{LS1}{}{}
\DeclareFontSubstitution{LS1}{stix}{m}{n}
\DeclareSymbolFont{stixletters}{LS1}{stix}{m}{it}
\DeclareMathAccent{\cev}{\mathord}{stixletters}{"91}
\DeclareMathAccent{\vec}{\mathord}{stixletters}{"92}
\DeclareMathAccent{\vecev}{\mathord}{stixletters}{"95}
\newcommand{\greedy}[1][k]{\text{$k$-\erase}}
\newcommand{\m}{\mathfrak{m}}
\makeatletter

\newcommand{\crist}[1][\@nil]{%
\def\tmp{#1}%
   \ifx\tmp\@nnil
       \lpcritical
    \else
       \lpcritical^{(#1)}
    \fi}
\newcommand{\infset}[1][\@nil]{
  \def\tmp{#1}%
   \ifx\tmp\@nnil
       \zeta
    \else
       \zeta^{#1}
    \fi}

\newcommand{\altinfset}[1][\@nil]{
  \def\tmp{#1}%
   \ifx\tmp\@nnil
       \widetilde{\zeta}
    \else
       \widetilde{\zeta}^{#1}
    \fi}
\newcommand{\binfset}[1][\@nil]{
  \def\tmp{#1}%
   \ifx\tmp\@nnil
       \cev{\zeta}
    \else
       \cev{\zeta}^{#1}
    \fi}
\makeatother

\newcommand{\Left}{\ensuremath{\mathtt{left}}}
\newcommand{\Right}{\ensuremath{\mathtt{right}}}
\newcommand{\Sleep}{\ensuremath{\mathtt{sleep}}}

\newtheorem{thm}{Theorem}
\newtheorem{lemma}[thm]{Lemma}
\newtheorem{prop}[thm]{Proposition}
\newtheorem{cor}[thm]{Corollary}

\newtheorem*{conjecture*}{Density Conjecture}

\theoremstyle{remark}

\theoremstyle{definition}

\definecolor{hancolor}{rgb}{0.0 0.0, 1.0}
\newcommand{\critFE}{\critical_{\mathtt{FE}}}
\newcommand{\critDD}{\critical_{\mathtt{DD}}}

\title[The hockey-stick conjecture for activated random walk]{The hockey-stick conjecture for\\ activated random walk}
\author{Christopher Hoffman}
 \address{Christopher Hoffman, Department of Mathematics, University of Washington}
 \email{\texttt{choffman@uw.edu}}
 \author{Tobias Johnson}
 \address{Tobias Johnson, Departments of Mathematics,  College of Staten Island, City University of New York}
\email{\texttt{tobias.johnson@csi.cuny.edu}}

 \author{Matthew Junge}
 	\address{Matthew Junge, Department of Mathematics, Baruch College, City University of New York}
	\email{\texttt{matthew.junge@baruch.cuny.edu}}

\begin{document}

\begin{abstract}
    We prove a conjecture of Levine and Silvestri that the driven-dissipative activated random walk model on an interval drives itself directly to and then sustains a critical density. This marks the first rigorous confirmation of a sandpile model behaving as in Bak, Tang, and Wiesenfeld's original vision of self-organized criticality.

\end{abstract}

\maketitle 

\section{Introduction}

 The theory of \emph{self-organized criticality} was introduced by Bak, Tang, and Wiesenfeld in the late 1980s to explain how critical-like behavior resembling lab-tuned phase transitions appears throughout nature without external tuning \cite{BakTangWiesenfeld87, BakTangWiesenfeld88}. Their seminal work has proven highly influential, with over ten thousand citations across diverse fields. The basic premise is that gradual tensioning among many components that is occasionally released in sudden bursts causes systems---such as snow slopes, tectonic plates, and star surfaces---to hover in complex states with power-law and fractal expressions. Their motivating illustration was of a growing sandpile on a table that reaches then sustains a critical slope once the sand begins to spill off the edges. Quoting \cite{BakTangWiesenfeld88},

 \begin{quote}
     To illustrate the basic idea of self-organized criticality in a transport system, consider a simple pile of sand. Suppose we start from scratch and build the pile by randomly adding sand, a grain at a time. The pile will grow, and the slope will increase. Eventually, the slope will reach a critical value (called the angle of repose); if more sand is added it will slide off.\,\ldots The critical state is an
     attractor for the dynamics.
 \end{quote}
 
 Bak, Tang, and Wiesenfeld introduced the \emph{abelian sandpile} as a model for self-organized criticality,
 emphasizing that they were ``interested in the general behavior of nonlinear diffusion dynamics\ldots
 and not in sand piles, per se.'' Based on simulations, they argued that the model
 behaved as in their illustration, increasing in density with the addition of particles
 until reaching a critical state. Crucially, they claimed that ``[o]nce the critical point is reached, 
 the system stays there.''
 
\begin{figure}\centering
    \begin{tikzpicture}[xscale=10/3,yscale=5/2]
      \draw[->] (-.03,0)--(3.1,0) node[right] {$\rho$};
      \draw[->] (0,-.06)--(0,2.1) node[above] {$D_\rho$};
      \draw[thick,blue] plot[smooth] file {nonuniversal.table};
      \foreach \x in {0,.5,...,3}
        \draw (\x,0) -- +(0,-.04) node[below,font=\small] {$\x$};
      \foreach \y in {0,.5,...,2}
        \draw (0,\y) -- +(-.03,0) node[left,font=\small] {$\y$};
      \draw (1.62,1.6) rectangle (1.92,1.71);
      \begin{scope}[shift={(1.8,.23)},xscale=4,yscale=80,shift={(-1.62,-1.6575)}]
        \draw (1.62,1.6575) rectangle (1.92,1.6725);
        \draw[thick,blue] plot[smooth] file {nonuniversalzoom.table};      
        \draw[dotted] (1.668898,1.6575)--(1.668898,1.6725);
        \draw[dotted] (1.62,1.668898)--(1.92,1.668898);
        \draw[dotted] (1.62,1.666666)--(1.92,1.666666);
        \foreach \x in {1.65,1.75,1.85}
          \draw (\x,1.6575) -- +(0,-.0005) node[below,font=\small] {$\x$};
        \draw (1.668898,1.6725) node[below,font=\small]{$1.6689$} (1.668898,1.6725) -- +(0,.0005) ;        
        \foreach \y in {1.66,1.665,1.67}
          \draw (1.62,\y) -- +(-.01,0) node[left,font=\small] {$\y$};
        \draw (1.92,1.668898) -- +(.01,0) node[right,font=\small] {$1.6689$};
        \draw (1.92,1.666666) -- +(.01,0) node[right,font=\small] {$\frac53$};
      \end{scope}
      
    \end{tikzpicture}
    \caption{\textbf{The limiting empirical density profile for the abelian sandpile
    on the flower graph as proven in \cite{fey2010approach}.}
    Though it appears at first glance to have the shape established in Theorem~\ref{thm:hockey},
    it in fact grows to a constant approximately equal to $1.6689$ before
    decreasing toward its asymptotic limit $5/3$. The numerical data given by Fey, Levine,
    and Wilson indicate similar behavior for the abelian sandpile on the two-dimensional lattice.
    } \label{fig:flower}
\end{figure}
 Many years later, Fey, Levine, and Wilson refuted this claim \cite{fey2010approach,fey2010driving}.
 Their evidence suggests that the two-dimensional abelian sandpile increases in density
 until it reaches the critical density for an infinite version of the model with a tuned phase transition,
 but then slowly decreases to a limiting density differing in the fourth decimal place.
 They rigorously proved this behavior for the abelian sandpile on several more tractable graphs
 (see Figure~\ref{fig:flower}). They identified this behavior as one of several examples
 of nonuniversality for the abelian sandpile caused by its slow mixing (see also
 \cite{jo2010comment,levine2015threshold,hough2019sandpiles,hough2021cutoff}).

 A sandpile model with stochastic evolution called \emph{activated random walk} (ARW) was introduced in \cite{dickman2010activated} and is believed to be an adequately universal model of self-organized criticality. 
 It can be formulated as an interacting particle system on a graph with active and sleeping  particles. Active particles perform simple random walk at exponential rate 1. When an active particle is alone, it falls asleep at exponential rate $\lambda \in (0,\infty)$. Sleeping particles remain in place but become active if an active particle moves to their site. If all active particles fall asleep, the system stabilizes with every site in the final configuration either empty or containing exactly one sleeping particle.

 In the \emph{driven-dissipative ARW with {uniform driving}}, active particles are added one at a time to a uniformly random site in a finite $d$-dimensional box viewed as a subgraph of $\mathbb Z^d$. The boundary edges lead to sinks that trap particles. In each step, the process runs after a particle is added until
it reaches a stable configuration with all particles either asleep or trapped at a sink.
Then in the next step, another driving particle is added and the process continues.
Viewed as a Markov chain on configurations of sleeping particles, the driven-dissipative version of ARW on a finite box is known to have a unique stationary distribution \cite{levine2021exact}. 
It was conjectured that the expected density of particles in the stationary distribution on a box
converges to a constant $\critDD=\critDD(d,\lambda)$ as the box size grows.

  The \emph{fixed-energy ARW} on $\ZZ^d$ starts with an ergodic initial configuration of density $\rho$. 
  This model has a critical density $\critFE= \critFE(d,\lambda)$:
  for $\rho < \critFE$, at each site activity ceases eventually while for $\rho > \critFE$ activity persists for all time \cite{rolla2019universality}. Much focus has been given to proving that $0<\critFE(d,\lambda) <1$ for all $d\geq 1$ and $\lambda >0$ \cite{rolla2012absorbing,SidoraviciusTeixeira17,StaufferTaggi18,BasuGangulyHoffman18,HoffmanRicheyRolla20,hu2022active,forien2022active,asselah2024critical}. 
    
 Dickman, Mu\~noz, Vespignani, and Zapperi clarified the theory of self-organized criticality by hypothesizing that  sandpile models should organize at the critical value corresponding to a phase transition in a conventional parametrized variant \cite{dickman1998self,dickman2000paths}. For example, their \emph{density conjecture} relates driven-dissipative ARW to fixed-energy ARW via the claim that $\critDD$ exists and is equal to $\critFE$  \cite{dickman1998self,rolla2020activated}.
 Fey, Levine, and Wilson's work refuted the density conjecture for the abelian sandpile model
 and demonstrated that its density does not evolve as predicted by Bak, Tang, and Wiesenfeld
 \cite{fey2010approach,fey2010driving}. (We note that these results are very convincing
 but are based on simulations, and it remains open to rigorously prove them.)
 We recently proved the density conjecture for ARW in dimension one \cite{hoffman2024density}, providing
 evidence for the universality of the model. It is thus a natural question
 whether driven-dissipative ARW follows Bak, Tang, and Wiesenfeld's prediction, with density
 growing to $\critFE$ and then remaining there.
 Levine and Silvestri called this the \emph{hockey-stick conjecture}
\cite[Conjecture 17]{levine2023universality}, after the shape of the limiting profile (see Figure~\ref{fig:hockey}).

We prove the hockey-stick conjecture for ARW in dimension one.
We emphasize that this behavior has never been rigorously established in any sandpile model.
In showing that ARW is attracted to a single critical density coinciding with that of a parameterized 
variant, we provide further theoretical evidence for Dickman et al.'s explanation of 
self-organized criticality.  Note that in our follow-up work \cite{hoffman2025cutoff}, we prove the stronger statement that driven-dissipative ARW with uniform driving converges to the stationary state after adding $\critFE n + o(n)$ particles as conjectured in \cite{levine2021exact}.

\begin{figure}
    \centering
    \begin{tikzpicture}[xscale=9,yscale=4.5]
      \draw[->] (-.03,0)--(1.3,0) node[right] {$\rho$};
      \draw[->] (0,-.06)--(0,1) node[above] {$D_\rho$};
      \foreach \x in {0, .25,.5,.75,.889,1,1.25}
        \draw (\x,0)--(\x,-.06) node[below] {$\x$};
      \foreach \y in {0,.25,.5,.75,.889}
        \draw (0,\y)--(-.03,\y) node[left] {$\y$};
      \draw plot[only marks,mark=*,mark size=.01pt,mark options={blue}] file {hockey2500-2000-.8.table};
    \end{tikzpicture}

  \caption{\textbf{A sample path of $D_\rho(n,\lambda)$ with $n=2000$ and $\lambda=.8$.}
    The graph shows the empirical density as $2500$ particles are added uniformly
    at random one at a time, stabilizing after each addition, on an interval of length~$2000$.
    The critical density $\critFE$ appears to be approximately $.889$.
    The sample path is a near-perfect match to the limiting empirical density
    profile $\rho\mapsto\min(\rho,\critFE)$. See \cite{levine2023universality} for a similar simulation in dimension two.}
    \label{fig:hockey}
\end{figure}

We state our result in more detail now. For a formal definition
and construction of ARW as a continuous-time Markov process, see \cite{rolla2020activated}.
 Consider driven-dissipative
 ARW with uniform driving on $\ii{1,n} :=\{1,\dots,n\}$ with sinks at $0$
 and $n+1$ that trap particles. Starting from an empty configuration, let $Y_{t}= Y_t(n,\lambda)$ be the number of sleeping particles in $\ii{1,n}$ after $t$ steps of the chain, 
 and let $D_\rho = D_\rho(n,\lambda) := Y_{\lceil \rho n \rceil}(n,\lambda) / n$ be the \emph{empirical density} of sleeping particles after $\lceil \rho n\rceil$ steps.
 Let $\critFE=\critFE(\lambda)$ be the critical density for fixed-energy
 ARW on $\ZZ$ with sleep rate $\lambda>0$. 
\begin{thm}\thlabel{thm:hockey}
For any $\lambda,\rho,\epsilon >0$,
\begin{align}\label{eq:hockey}
  \P\Bigl(\abs[\big]{D_\rho(n,\lambda) - \min(\rho,\critFE(\lambda))}>\epsilon\Bigr)\leq Ce^{-cn}
\end{align}
for constants $c,C>0$ depending only on $\lambda$ and $\epsilon$.
\end{thm}

The conjecture as originally proposed by Levine and Silvestri was that 
$D_\rho(n,\lambda)\to\min(\rho,\critFE(\lambda))$ in probability as $n\to\infty$
for any fixed $\lambda$ and $\rho$. Our version is stronger and implies
that $D_\rho$ and $\min(\rho,\critFE)$ are close in $\sup$-norm over a growing interval:

\begin{cor}\thlabel{cor:sup.norm}
  For all $\epsilon>0$, there exists some
  $\alpha=\alpha(\epsilon,\lambda)>1$ such that
  \begin{align*}
    \P\biggl( \sup_{0<\rho<\alpha^n}\abs[\big]{D_\rho-\min(\rho,\critFE)} > \epsilon \biggl)\to 0.
  \end{align*}  
\end{cor}

To prove \thref{thm:hockey}, results from \cite{hoffman2024density} immediately
show that $D_\rho$ is unlikely to be much larger than $\min(\rho,\critFE)$.
Our work comes in showing that it is unlikely to be smaller, which follows from
the following proposition stating that only a small quantity of particles
exit the interval when stabilizing a subcritical or critical
quantity of randomly placed particles.
 Define a \emph{configuration} of particles $\sigma\in \{\s, 0, 1, 2, \hdots\}^{\ii{1,n}}$
to consist of counts of active particles at each site, with $\s$ denoting a lone sleeping particle.
\begin{prop}\thlabel{thm:lower.bound}
  Let $0<\rho\leq\critFE$ and let
  $\sigma$ consist of $\ceil{\rho n}$ 
  active particles placed independently and uniformly at random on $\ii{1,n}$.
  Let $M$ be the maximum of the number of particles ejected to the left and right sinks when
  stabilizing $\sigma$. 
  For $\epsilon>0$, there exist constants $c,C>0$ that only depend on $\epsilon$ and $\lambda$ such that
  \begin{align}\label{eq:lower.bound.ind}
    \P\bigl(M(\sigma)\leq \epsilon n/2 \bigr) \geq 1 - C e^{-cn}.
  \end{align}
\end{prop}

The proof of \thref{thm:lower.bound} uses a general theory that was developed in 
\cite{hoffman2024density} to prove the density conjecture.
In the standard \emph{sitewise representation} of ARW, we view the particles as moving according
to stacks of random instructions placed on the sites of the graph. An \emph{odometer} counts
the number of instructions executed at each site. The \emph{stabilizing odometer} on the interval
is the one produced when the system runs to stability, with all particles sleeping or absorbed
at the sink. 
According to the \emph{least-action principle},
the stabilizing odometer is minimal in the class of \emph{stable odometers}, which are defined as
respecting mass-balance equations at each site.
Thus, any stable odometer gives an upper bound on the true stabilizing
odometer and hence an upper bound on the loss of particles when stabilizing.
For precise definitions of these terms, we refer the reader to
the standard reference \cite{rolla2020activated} and to the appendix.
To produce a suitable stable odometer, we use the theory from \cite{hoffman2024density},
which embeds each stable odometer as an \emph{infection path} in
a directed $(2+1)$-dimensional process called \emph{layer percolation}.

The difficulty in our proof comes
in this final step of producing a stable odometer using layer percolation.
In \cite{hoffman2024density}, we carry this out for initial configurations either of
a single particle at all sites or of a mass of particles all on a single site.
These constructions are delicate and rely on these regular
starting configurations. To prove the hockey-stick conjecture, we need
to consider initial configurations of randomly placed particles. To do so,
we develop a more robust technique to get upper bounds on the stabilizing odometer.
The main idea is to construct two separate odometers for applying the least-action principle,
one to get a bound at the left endpoint and one at the right, thus avoiding the difficulty
of producing a single stable odometer yielding near-optimal bounds at both endpoints simultaneously.
This improvement on the methods of \cite{hoffman2024density} is of independent interest
and should prove useful for producing sharp bounds on the stabilizing odometer in a variety
of circumstances.

In the next section, we prove \thref{thm:hockey,cor:sup.norm} under the assumption of
\thref{thm:lower.bound}. Then in the final section we give the proof of \thref{thm:lower.bound} and describe
in more detail how its techniques differ from previous ones. Because we rely so heavily
on the new and lengthy paper \cite{hoffman2024density}, we have included appendices extracting important
definitions and results for easy reference.

\section{Proof of the hockey-stick conjecture}

We start with the derivation of \thref{thm:hockey} from \thref{thm:lower.bound}.
According to the definition of the driven-dissipative Markov chain, the density $D_\rho$
is found by adding $\ceil{\rho n}$ particles in sequence, stabilizing the system after each addition.
We will usually take an alternate view of $D_\rho$ as the density after adding all
$\ceil{\rho n}$ particles and stabilizing the system once.
This procedure takes us to the same final configuration when using the sitewise representation
by the abelian property of ARW (see the appendix or \cite[Lemma~2.4]{rolla2020activated}).

    \begin{proof}[Proof of \thref{thm:hockey}]
      For the upper bound on $D_\rho$, we first observe that for $\rho\leq\critFE$ we have
      $D_\rho\leq \ceil{\rho n}/n\leq \rho+1/n$ deterministically.
      For $\rho>\critFE$, we apply \thref{peanut butter} to show that 
      $\P(D_{\rho}> \critFE + \epsilon) \leq C e^{-cn}$ for constants $C$ and $c$ that 
      depend only on $\epsilon$ and $\lambda$.
      
      For the lower bound, if $\rho\leq\critFE$ then \thref{thm:lower.bound} directly shows
      that $\P(D_{\rho}< \rho - \epsilon) \leq C e^{-cn}$ for constants $C$ and $c$ that 
      depend only on $\epsilon$ and $\lambda$.
      To handle the $\rho>\critFE$ case, we note that 
      $D_\rho$ is stochastically increasing in $\rho$ \cite[Lemma~7]{forien2025newproof}.
      Hence for $\rho>\critFE$, we have
      \begin{align*}
        \P(D_\rho<\critFE-\epsilon) &\leq \P(D_{\critFE}< \critFE - \epsilon) \leq C e^{-cn}
      \end{align*}
      by the already established $\rho=\critFE$ case.
    \end{proof}

\begin{proof}[Proof of \thref{cor:sup.norm}]
  Since $D_\rho$ is a step function jumping at $\rho=k/n$ for integers $k$, it is enough to control
  it at these points. To do so, we
  apply \thref{thm:hockey} at $\rho=k/n$ and take a union bound over all integers $0<k<n\alpha^n$,
  choosing $\alpha>1$ small enough so that the exponential term in \eqref{eq:hockey} dominates.
\end{proof}

\section{Proof of Proposition~\ref{thm:lower.bound}}

Now comes the main task of this paper, applying the theory from \cite{hoffman2024density} to prove
\thref{thm:lower.bound}.
To distill this theory into a few sentences, it embeds the \emph{stable odometers}
for ARW as \emph{infection paths} in a $(2+1)$-dimensional directed process
we call \emph{layer percolation} (see Appendix~\ref{sec:notation} for definitions).
It is established using this connection that the critical density $\critFE$ for ARW
is equal to the growth rate $\crist$ for the height of the infected set in layer percolation.
We have tried to make this section comprehensible without close familiarity with
\cite{hoffman2024density} and have included appendices summarizing its definitions and results.
For an in-depth guide to the theory, we refer the reader to \cite{hoffman2024density},
especially its introduction and the examples in its Section~3 demonstrating the embedding.

To prove \thref{thm:lower.bound}, we must construct a stable odometer and apply
the least-action principle (\thref{lem:lap}) to establish an accurate
upper bound on the true stabilizing odometer for ARW with initial configuration
given by $\ceil{\rho n}$ randomly placed particles.
We can produce such a stable odometer by finding an infection path in layer percolation
and translating it back into an odometer, but there are two difficulties.
First, the odometer produced from a length~$n$ infection path is not necessarily
stable at the final site~$n$. In \cite{hoffman2024density}, we use a powerful
result \cite[Lemma~6.1]{hoffman2024density} to show that layer percolation
does contain infection paths corresponding to odometers stable at $n$.
Second, infection paths in layer percolation are in bijection not with stable odometers but with a larger
class of functions called \emph{extended stable odometers}.
These functions are allowed to take negative values that have no real meaning from the perspective
of ARW. Thus, after producing an infection path in layer percolation and translating it over to
an extended stable odometer, we must prove that it takes only nonnegative values.
This is done in \cite[Propositions~8.5 and 8.7]{hoffman2024density} using a method
that relies crucially on the initial configuration being a single contiguous block of particles,
which is satisfied by the two initial configurations---a single particle at all sites and a collection
of particles all at a single site---that are considered there. But the method fails for the
initial configurations in this paper.

We address both of these difficulties at once with an improved technique.
We use layer percolation to construct an extended odometer not on $\ii{1,n}$
but on the larger interval $\ii{0,n+1}$. The resulting odometer is automatically
stable on the interior $\ii{1,n}$ without additional assumptions.
While it is an odometer on a larger interval than $\ii{1,n}$, we can still
use it to apply the least-action principle, once it is shown to take nonnegative values.
One can think of such an odometer as representing a stabilizing odometer
for the system with extra particles injected at the endpoints.

Constructing the odometer on $\ii{0,n+1}$ also makes it easier to establish its nonnegativity.
Freed of the constraint of making the odometer stable at its right endpoint,
we can instead make it as large as possible there. Because it is larger, it is easier
to prove it nonnegative. The cost is that the odometer will be far above the
stabilizing odometer at the right endpoint and would yield a weak bound there.
But at the left endpoint our bound will still be accurate, and by symmetry
the bound applies to the right endpoint as well.
Thus this improved technique yields accurate bounds on the odometer at both endpoints,
thereby bounding the count of particles lost to the sink.

We turn to the proofs now. We will say than an event holds \emph{with overwhelming probability} (w.o.p.)\ if
its failure probability is bounded by $Ce^{-cn}$ for constants $c,C>0$ with possible dependencies
to be specified.

In the bijection between extended stable odometers and infection
paths, a key role is played by the
\emph{minimal odometer} of $\eosos{n+1}(\Instr,\sigma,u_0,f_0)$,
which is the minimal extended
odometer on $\ii{0,n+1}$ stable on $\ii{1,n}$ taking prescribed values at sites $0$ and $1$
(see Appendix~\ref{sec:notation}). We note that the minimal odometer is not the same
as the true stabilizing odometer, which is minimal over a different class of objects. 
We start with a lower bound on the minimal odometer:

    \begin{lemma}\thlabel{lem:min.bounds}
      Let $\sigma$ consist of $\ceil{\rho n}$ active particles placed independently
      and uniformly at random on $\ii{1,n}$ for $\rho\in(0,\critFE]$.
      Let $\rho'\leq\rho$, and
      let $\m$ be the minimal odometer of $\eosos{n+1}(\Instr,\sigma,u_0,f_0)$
      with $u_0=0$ and $f_0=-\floor{(\rho-\rho')n/2}$.
      For any $j\in\ii{1,n+1}$ and $\delta>0$,
      \begin{align*}
        \P\Biggl(\rt{\m}{j} \geq  \frac{(\rho-\rho')jn - \rho j^2}{2}  -\delta n^2\Biggr)
        \geq 1- Ce^{-cn}
      \end{align*}
      for constants $c,C>0$ that depend only on $\delta$.
    \end{lemma}
    
    \begin{proof}
      Here we allow the constants for overwhelming probability to depend on $\delta$.
      Let $Z_i=\sum_{v=1}^i\sigma(v)$, the number of particles initially placed in
      $\ii{1,i}$.
      By \thref{prop:min.odometer.concentration},
      \begin{align}\label{eq:modcb}
        \abs[\Bigg]{\rt{\m}{j} -  \sum_{i=1}^j (-f_0 - Z_i) }
          &\leq \delta n^2/2\text{ w.o.p.}
      \end{align}
      Thus our task is to prove
      \begin{align}\label{eq:focus.bound}
         \sum_{i=1}^j (-f_0 - Z_i) \geq
             \frac{(\rho-\rho')jn-\rho j^2}{2} - \delta n^2/2\text{ w.o.p.}
      \end{align}
      For $1\leq i\leq n$, we have
      $Z_i\sim \Bin(\lceil \rho n\rceil ,\, i / n )$ for $1\leq i\leq n$,
      and $Z_{n+1}=Z_n$. Thus $\E Z_i=\rho i + O(1)$ for all $i\in\ii{1,n+1}$, and
      by a Chernoff bound for each $i$ we have
      $Z_i\leq\rho i+ \delta n /5$ w.o.p.
      By a union bound, $\abs[\big]{\sum_{i=1}^j Z_i - \rho j^2/2}\leq \delta n^2/4$ w.o.p.
      Using these bounds together with $f_0=-(\rho-\rho')n/2+O(1)$, we have proven \eqref{eq:focus.bound}.
      And \eqref{eq:modcb} and \eqref{eq:focus.bound} together prove the lemma.
    \end{proof}

    We are ready for the proof of \thref{thm:lower.bound} now. The idea is to consider
    the class of extended stable odometers on $\ii{0,n+1}$ executing zero instructions at site~$0$
    (which is the sink) and executing exactly $\floor{\epsilon n/2}$
    \Left\ instructions at site~$1$, our desired upper bound there. We use layer percolation
    to produce an extended stable odometer in this class that grows as rapidly as possible.
    We then use this growth to prove that the extended stable odometer takes nonnegative values
    and hence bounds the stabilizing odometer via the least-action principle.
    \begin{proof}[Proof of \thref{thm:lower.bound}]
      We allow the constants in overwhelming probability bounds to depend on $\epsilon$
      and $\lambda$.
      Let $\rho' = \rho-\epsilon$.
      We will show that the odometer produced by stabilizing $\sigma$
      on $\ii{1,n}$ executes at most $\epsilon n/2$ \Left\ instructions at site~$1$ w.o.p.
      By symmetry the same statement is true for \Right\ instrucions at site~$n$,
      thus proving the proposition.
      
      Let $f_0=-\floor{\epsilon n/2}$ and $u_0=0$,
      and let us abbreviate $\eosos{n+1}(\Instr,\sigma,u_0,f_0)$
      as $\eosos{n+1}$.
      By definition, all extended odometers $u\in\eosos{n+1}$ satisfy
      $\rt{u}{0}-\lt{u}{1}=f_0$ and $u(0)=0$. Hence they
      execute $-f_0=\floor{\epsilon n/2}$ \Left\ instructions at site~$1$.
      Also by definition, they are stable on $\ii{1,n}$. Thus, if we can produce
      any genuine odometer (i.e., taking nonnegative values) in $\eosos{n+1}$,
      then the least-action principle (\thref{lem:lap}) applies and shows that the true
      odometer stabilizing $\sigma$ on $\ii{1,n}$ ejects at most $\floor{\epsilon n/2}$
      particles from the left endpoint. We show now that it is likely we can construct
      such an odometer.

      Take $\rho''=\rho-\epsilon/2$ and $\rho'''=\rho-\epsilon/4$, so that
      \begin{align*}
        \rho'<\rho''<\rho'''<\rho\leq\crist.
      \end{align*}
      Recall the definitions of the $k$-greedy path and $\crist[k]$ (see the appendix)
      and choose $k=k(\lambda,\epsilon)$ large enough that $\crist[k]\geq \rho'''$.
      Let $(0,0)_0=(r_0,s_0)_0\to\cdots\to(r_{n+1},s_{n+1})_{n+1}$ be the $k$-greedy infection
      path in $\ip{n+1}(\Instr,\sigma,u_0,f_0)$, and let $u$ be an extended odometer
      in $\eosos{n+1}(\Instr,\sigma,u_0,f_0)$ corresponding to it. Recall that under this correspondence,
      $r_j$ counts the additional \Right\ instructions executed by $u$ at site~$j$ beyond
      the minimal odometer (i.e., $\rt{u}{j} = \rt{\m}{j}+r_j$), and $s_j$ counts the number
      of sites in $\ii{1,j}$ where the final instruction executed under $u$ is \Sleep.
      
      Now we demonstrate that $u(j)\geq 0$ for $j\in\ii{0,n+1}$.
      For small $j$, this will hold by an application of \thref{lem:nonnegative}.
      The idea is that $u$ must execute a positive number of \Left\ instructions
      at sites near $0$ to create a leftward flow of particles
      so that the requisite quantity $-f_0$ finish at $0$. For larger $j$, we can deduce $u(j)\geq 0$
      from lower bounds on $\rt{\m}{j}$ and $r_j$ given by \thref{lem:min.bounds,prop:greedy.path}.
      
      We start with the case of small $j$. Let $\alpha=\min(\epsilon/3\rho,1)$.
      We will argue that $u(j)\geq 0$ for $0\leq j\leq\alpha n$ w.o.p.
      Let $Z_i=\sum_{v=1}^i\sigma(v)$, the number of particles initially in
      $\ii{1,i}$.
      We claim that       
      \begin{align}
        Z_{\floor{\alpha n}} \leq (1.1)\rho\alpha n\text{ w.o.p.}\label{eq:Z.lower}
      \end{align}
      This statement follows by applying
      Hoeffding's inequality to $Z_{\floor{\alpha n}}$,
      which has distribution $\Bin(\ceil{\rho n},\floor{\alpha n}/n)$,
      to obtain
      \begin{align*}
        \P\bigl(Z_{\floor{\alpha n}} > (1.1)\rho\alpha n\bigr)
          \leq \exp\bigl(-c(\rho\alpha n)^2/\rho n\bigr) =\exp(-c \epsilon^2n/9\rho)
      \end{align*}
      for some absolute constant $c$. This proves \eqref{eq:Z.lower}
      since $\rho\leq 1$.
      
      For $i\geq 1$ let
      \begin{align}\label{eq:f_i}
        f_i:=\rt{u}{i}-\lt{u}{i+1}
      \end{align}
      be the flow of particles from $i$ to $i+1$ as in
      \thref{lem:flow,lem:nonnegative}, and note that \eqref{eq:f_i} holds for $i=0$
      as well by definition of $\eosos{n+1}$.
      For $0\leq j\leq\alpha n$, we have $Z_j\leq Z_{\floor{\alpha n}}$.
      By the stability of $u$ on $\ii{1,n}$, we have
      $f_j=f_0+Z_j-s_j\leq f_0+Z_j$ by \thref{lem:flow}.
      By \eqref{eq:Z.lower} and our choice of $\alpha$, we obtain $f_j\leq f_0+Z_j<0$ for 
      $0\leq j\leq\alpha n$ w.o.p. Since $u(0)=0$, repeated application
      of \thref{lem:nonnegative}\ref{i:nonnegative.b} proves that $u(j)\geq 0$ for 
      $0\leq j\leq\alpha n$ w.o.p.
      
      Next, we consider $j\geq\alpha n$.
      Since we chose $\crist[k]\geq\rho'''=\rho''+\epsilon/4$, we can apply \thref{prop:greedy.path}
      with $t=\epsilon\sqrt{j}/8$ to show that 
      \begin{align*}
        \P\biggl( r_j < \frac{\rho'' j^2}{2} \biggr)
          &\leq \P\Biggl( r_j<\frac{\bigl(\crist[k] - \epsilon/4\bigr)j^2}{2}\Biggr)
          \leq C\exp\biggl(-\frac{c\epsilon^2 j / 64}{1+\epsilon/8}\biggr)
      \end{align*}
      for constants $c,C>0$ depending only on $\lambda$ and $k=k(\lambda,\epsilon)$.
      Since $\alpha$ is bounded away from zero by a quantity depending only on $\epsilon$,
      we have $r_j\geq \rho'' j^2/2$ w.o.p.\ for $j\geq \alpha n$.
      Combining this bound with \thref{lem:min.bounds}, for all $\alpha n\leq j \leq n+1$ and any fixed $\delta>0$
      it holds with overwhelming probability that
      \begin{align*}
        \rt{u}{j} =r_j+\rt{\m}{j} 
          &\geq \frac{\rho'' j^2}{2} + \frac{(\rho-\rho')jn-\rho j^2}{2}-\delta n^2\\
            &= \tfrac12 j\bigl((\rho-\rho')n-(\rho-\rho'')j\bigr)-\delta n^2\\
              &\geq \tfrac12 j\bigl((\rho-\rho')n-(\rho-\rho'')(n+1)\bigr)-\delta n^2\\
              &= \tfrac12 jn\bigl( \rho'' -\rho' -\tfrac{1}{n}(\rho-\rho'')\bigr)-\delta n^2.
      \end{align*}
      Now, choose a large enough constant $n_0=n_0(\epsilon)$ and assume $n\geq n_0$
      to make $\rho''-\rho'-\frac1n(\rho-\rho'')$ bounded from below, say by $\epsilon/4$.
      Applying the bound $j\geq\alpha n$ and choosing $\delta=\delta(\epsilon)$
      small enough, we obtain $\rt{u}{j}>0 $ for all $\alpha n\leq j\leq n+1$ w.o.p.,
      showing that $u(j)>0$ for $\alpha n\leq j\leq n+1$ w.o.p.
      Together with the previous case, this shows that $u(j)\geq 0$ for all $j\in\ii{0,n+1}$
      w.o.p.
      
      Thus, we have shown that with overwhelming
      probability there exists an odometer $u$ on $\ii{0,n+1}$ that is stable on $\ii{1,n}$
      and satisfies $\lt{u}{1}=\floor{\epsilon n/2}$.
      Let $s$ be the true odometer stabilizing $\ii{1,n}$.
      By \thref{lem:lap}, we have $\lt{s}{1}\leq\floor{\epsilon n/2}$ w.o.p.
      And by symmetry, $\rt{s}{n}\leq\floor{\epsilon n/2}$ w.o.p.
    \end{proof}

\section*{Acknowledgments}
We thank Joshua Meisel for pointing out that it is simpler to start layer percolation
at the left sink rather than the left endpoint.
Hoffman was partially supported by NSF DMS Awards 1954059 and 2503778.
Johnson was partially supported by NSF DMS Award 2503779 and gratefully acknowledges support from
Uppsala University and the Wenner--Gren Foundation. 
Junge was partially supported by NSF DMS Award 2238272. 

\appendix

\section{Selected notation and definitions from \texorpdfstring{\cite{hoffman2024density}}{[HJJ24]}}
\label{sec:notation}

Below are excerpts from \cite[Sections~2--4]{hoffman2024density}. 

\begin{description}
    \item[$\Instr$] Assign each site $v$ a list of \emph{instructions} 
consisting of the symbols \Left, \Right, and \Sleep. We write $\Instr_v(k)$ to denote the $k$th instruction
at site~$v$, and we take $\bigl(\Instr_v(k);\ v \in \mathbb Z,\, k\geq 1\bigr)$ to be i.i.d.\ with
\begin{align*}
  \Instr_v(k)=\begin{cases} \Left&\text{with probability $\frac{1/2}{1+\lambda}$,}\\
    \Right&\text{with probability $\frac{1/2}{1+\lambda}$,}\\
    \Sleep&\text{with probability $\frac{\lambda}{1+\lambda}$.}
  \end{cases}
\end{align*}
    \item[Odometer] An odometer $u$ counts how many instructions $u(v)$ are executed at each site $v$.
      The number of \Right\ and \Left\ instructions used at $v$ is denoted, respectively,
      as $\rt{u}{v}$ and $\lt{u}{v}$. That is, $\rt{u}{v}$ and $\lt{u}{v}$ give the number of \Right\ and
      \Left\ instructions, respectively, among $\Instr_v(1),\ldots,\Instr_v(u(v))$.
      We view all odometers formally as nonnegative functions on $\ZZ$, and we say that $u$ is
      an odometer on an interval $\ii{a,b}$ to mean that it takes the value zero off $\ii{a,b}$.
    \item[The true odometer stabilizing a finite set of sites $V$] Given $\Instr$ and $\sigma$, 
      this is the odometer after ARW executes on $V$ with particles trapped on exiting $V$.
    \item[Abelian property]  All sequences of topplings within a finite set $V$ that stabilize $V$ have the same odometer.
    \item[Stable odometer] Let $u$ be an odometer on $\ZZ$ and 
  let $\sigma$ be an ARW configuration with no sleeping particles.
  We call $u$ \emph{stable on $V \subseteq \ZZ$} for the initial configuration $\sigma$ and 
  instructions $\bigl(\Instr_v(i),\,v\in \ZZ,\,i\geq 1\bigr)$ if for all $v\in V$,
  \begin{enumerate}[(a)]
    \item $h(v) := \sigma(v)+\rt{u}{v-1} + \lt{u}{v+1} - \lt{u}{v} - \rt{u}{v} \in\{0,1\}$;\label{i:balance}
    \item $h(v)=1$ if and only if $\Instr_v(u(v))=\Sleep$.
      \label{i:end.in.sleep}
  \end{enumerate}
  \thref{lem:lap} states that the true odometer stabilizing a finite set of sites $V$ 
  is the minimal odometer on $\ZZ$ stable on $V$.
    \item[Extended odometer]  We extend each stack of instructions to be two-sided by making
      defining $\Instr_v(i)$ for $i < 0$ as i.i.d.\ instructions
      and fixing $\Instr_v(0)=\Left$ for $v\in\mathbb Z$.
Extended odometers are maps $\ZZ\to\ZZ$, with negative values representing 
the execution of instructions on the negative-index
portion of the instruction list, but with the reverse of their normal effects (e.g., a \Left\ instruction
executed at site $v$ pulls a particle from $v-1$ to $v$ rather than pushing one from $v$ to $v-1$).
The point of this extension is to set up the correspondence between odometers and layer percolation
(to be described), but its details are unimportant for this paper.

For the interested reader, 
the counts $\lt{u}{v}$ and $\rt{u}{v}$  are defined
for an extended odometer $u$ with $u(v)<0$ as the negative of the number of \Left\ and \Right\ instructions
at indices $v+1,\ldots,0$. The definition of \emph{stable odometer} is extended verbatim
to extended odometers.
The effect of these extensions is to make it so that for any $k\in\ZZ$, the set
of indices $i$ with $\lt{v}{i}=k$ is some interval $\ii{a,b}$,
and the instructions $(\Instr_v(a),\ldots,\Instr_v(b))$ has the same distribution
for any $k$. We set $\Instr_v(0)=\Left$ to avoid a size-biasing issue that would distinguish
the case $k=0$.

\item[$\eosos{n}(\Instr,\sigma,u_0,f_0)$]
  The set $\eosos{n}(\Instr,\sigma,u_0,f_0)$ consists of all extended odometers $u$ on $\ii{0,n}$ 
  stable on $\ii{1,n-1}$ for a given set of instructions $\Instr$ and 
  initial configuration $\sigma$ and that satisfy $u(0)=u_0$ and have 
  net flow $f_0$ from site~$0$ to site~$1$, i.e.,
  $\rt{u}{0} - \lt{u}{1}=f_0$. This set of extended odometers will have a correspondence
  with layer percolation.

    \item[Minimal odometer $\m$ of $\eosos{n}(\Instr,\sigma,u_0,f_0)$] The minimal 
      element (pointwise) of the extended odometers $\eosos{n}(\Instr,\sigma,u_0,f_0)$,
      which can be obtained by the following inductive procedure.
First let $\m(0)=u_0$. Now suppose that $\m(v-1)$ has already been defined. Then
$\m(v)$ is the minimum integer such that $\lt{\m}{v} = \rt{\m}{v-1} - f_0 - \sum_{i=1}^{v-1}\abs{\sigma(i)}$.
\emph{Minimal odometer} is a slight misnomer, in that $\m$ is in general
only an extended odometer and may take negative values.

We take a moment to distinguish the minimal odometer of $\eosos{n}(\Instr,\sigma,u_0,f_0)$
from the true stabilizing odometer of $\ii{1,n-1}$.
The former is the minimal element of $\eosos{n}(\Instr,\sigma,u_0,f_0)$,
a set of extended odometers stable on $\ii{1,n-1}$
with a condition on the left boundary. The latter is 
by the least-action principal (\thref{lem:lap}) the minimal element of the stable odometers
on $\ii{1,n-1}$, a set which has no boundary constraint
 but contains only non-extended odometers. Thus in general,
neither odometer is bounded by the other. In the typical case we consider in this paper,
the minimal odometer of $\eosos{n}(\Instr,\sigma,u_0,f_0)$ is greater than the true stabilizing
odometer near the left endpoint but is below it elsewhere and takes negative values at sufficiently
large sites.

    \item[Layer percolation] A sequence $(\zeta_k)_{k\geq 0}$ of subsets of $\NN^2$.
We think of a point $(r,s)_k\in\zeta_k$ as a \emph{cell} in column~$r$ and row~$s$ 
at step~$k$ of layer percolation that has been \emph{infected}.
At each step, every cell infects cells in the next step at random
as described in \cite[Section~3.2]{hoffman2024density}. 
We take $\zeta_0=\{(0,0)_0\}$
and then let $\zeta_{k+1}$
consist of all cells infected by some cell in $\zeta_k$.
An \emph{infection path} is a (typically finite) sequence of cells, 
each of which infects its successor.
\item[$\ip{n}(\Instr,\sigma,u_0,f_0)$]
Layer percolation can be coupled with ARW in a natural way by constructing it
from $\Instr$, $\sigma$, $u_0$, and $f_0$.
The set $\ip{n}(\Instr,\sigma,u_0,f_0)$ consists of all infection paths of length~$n$ starting from
$(0,0)_0$ in this coupled realization of layer percolation.
\item[The correspondence between odometers and layer percolation] 
Loosely speaking, extended odometers in $\eosos{n}(\Instr,\sigma,u_0,f_0)$
correspond to infection paths in $\ip{n}(\Instr,\sigma,u_0,f_0)$.
The ending row the infection path---i.e., $s$ where $(r,s)_n$ is the final cell
of the infection path---corresponds to the number of particles left sleeping by
the extended odometer on $\ii{1,n}$.
      
Formally, the correspondence is given by \cite[Proposition~4.6]{hoffman2024density}
and states that there is a surjective map $\Phi\colon \eosos{n}(\Instr,\sigma,u_0,f_0)\to\ip{n}(\Instr,\sigma,u_0,f_0)$ taking
extended stable odometers to infection paths,
with $\Phi(u)$ for $u\in\eosos{n}(\Instr,\sigma,u_0,f_0)$ defined as the sequence of cells
 $\bigl((r_v,s_v)_v,\,0\leq v\leq n\bigr)$ given by
\begin{align*}
  r_v &= \rt{u}{v} - \rt{\m}{v},\\
  s_v &= \sum_{i=1}^v\1\{\Instr_i(u(i))=\Sleep \}.
\end{align*}
Here $\m$ is the minimal odometer for $\eosos{n}(\Instr,\sigma,u_0,f_0)$.
The map $\Phi$ becomes injective if we consider strings of consecutive \Sleep\ instructions
to be the same, i.e., we identify elements $u,u'\in\eosos{n}(\Instr,\sigma,u_0,f_0)$
if for all $v\in\ii{0,n}$, either $u(v)=u'(v)$ or $u(v)$ and $u'(v)$
are indices in a string of consecutive \Sleep\ instructions.
     \item[The $k$-greedy path] This is a sequence of cells $(r_0,s_0)_0,\,(r_1,s_1)_1,\ldots$
 defined by the following inductive procedure:
 Starting from $(r_0,s_0)_0=(0,0)_0$, choose some cell $(r_k,s_k)_k$ that is infected
 starting from $(r_0,s_0)_0$
 with $s_k$ maximal. Let $(r_0,s_0)_0\to\cdots\to(r_k,s_k)_k$ be any infection path leading to $(r_k,s_k)_k$.
 Then choose $(r_{2k},s_{2k})_{2k}$ to be some cell infected starting from
 $(r_k,s_k)_k$ with $s_{2k}$ maximal,
 and take $(r_k,s_k)_k\to\cdots\to(r_{2k},s_{2k})_{2k}$ to be any infection path
 from $(r_k,s_k)_k$ to $(r_{2k},s_{2k})_{2k}$, and so on.
 The choice of $(r_{jk},s_{jk})_{jk}$ in each step of the process and the choice of infection path
 $(r_{(j-1)k},s_{(j-1)k})_{(j-1)k}\to\cdots\to(r_{jk},s_{jk})_{jk}$ is not important to us, 
 so long as it only depends on information up to step~$jk$ of layer percolation. The quantity $\crist^{(k)}$ is the expected increase in the $s$-coordinate after each step in the $k$-greedy path.
\item[The critical density $\crist$] 
  Defined as
  $\crist=\limsup_{k\to\infty}\crist[k]=\limsup_{k\to\infty}\frac1k\E X_k$,
  where $X_k$ is the highest row infected at step~$k$ of layer percolation starting from $(0,0)_0$.
\end{description}

\section{Selected results from \texorpdfstring{\cite{hoffman2024density}}{[HJJ24]}}

These are restatements of the indicated results from \cite{hoffman2024density}.

\begin{thm}[Theorem 8.4]\thlabel{peanut butter}
  Consider activated random walk with sleep rate $\lambda>0$.
  Let $\sigma$ be an initial configuration on $\ii{1,n}$ with no sleeping particles.
  Let $Y_n$ be the number of particles left sleeping on $\ii{1,n}$ in the stabilization of 
  $\sigma$ on
  $\ii{1,n}$. For any $\rho>\crist(\lambda)$,
  \begin{align*}
    \P(Y_n\geq \rho n) \leq Ce^{-cn}
  \end{align*}
  where $C,c$ are positive constants depending on $\lambda$ and $\rho$ but not on $n$ or $\sigma$.
\end{thm}

\begin{prop}[Proposition 8.5] \thlabel{prop:driven.dissipative.lower}
  Let $S_n$ be distributed as the number of sleeping particles under the invariant
  distribution of the driven-dissipative Markov chain on $\ii{1,n}$.
  For any $\rho<\crist(\lambda)$,
  \begin{align*}
    \P( S_n\leq\rho n) \leq Ce^{-cn}
  \end{align*}
  for constants $c,C$ depending only on $\lambda$ and $\rho$.
\end{prop}

\begin{prop}[Proposition 5.8]\thlabel{prop:min.odometer.concentration}
 Let  
  $\m$ be the minimal odometer of $\eosos{n}(\Instr,\sigma,u_0,f_0)$.
  Let
  \begin{align*}
    e_i= -f_0-\sum_{v=1}^i\abs{\sigma(v)}
  \end{align*}
  and suppose that $\abs{e_i}\leq\emax$ for some $\emax\geq 1$.
  For some constants $c,C>0$ depending only on $\lambda$, it holds
  for all $t\geq 4\emax$ that
  \begin{align}\label{eq:min.odometer.bound}
    \P\Biggl(\abs[\bigg]{\rt{\m}{j}-\biggl(\frac{u_0}{2(1+\lambda)}+\sum_{i=1}^je_i\biggr)} \geq t\Biggr)
      \leq C\exp\biggl(-\frac{ct^2}{n\bigl(n\emax + u_0+t\bigr)}\biggr)
  \end{align}
  for all $1\leq j\leq n$.
\end{prop}

\begin{lemma}[Lemma~4.1]\thlabel{lem:flow}
  Let $u$ be an extended odometer on $\ii{0,n}$.
  Let $f_v=\rt{u}{v}-\lt{u}{v+1}$, the net flow from $v$ to $v+1$.
  Let $s_v=\sum_{i=1}^v\1\{\instr_i(u(i))=\Sleep\}$.
  Then $u$ is stable on $\ii{1,n-1}$ for initial configuration $\sigma$ if and only if
  \begin{align}\label{eq:lemflow}
    f_v = f_0 + \sum_{i=1}^v\abs{\sigma(i)} - s_v\quad\text{for all $v\in\ii{0,n-1}$.}
  \end{align}
\end{lemma}

  \begin{lemma}[Contrapositive of Lemma 8.2]\thlabel{lem:nonnegative}
    Suppose that $u$ is an extended odometer on $\ii{0,n}$ stable
    on $\ii{1,n-1}$.
    Let $f_v=\rt{u}{v}-\lt{u}{v+1}$, the net flow from $v$ to $v+1$.
    \begin{enumerate}[(a)]
    \item For any $v\in\ii{0,n-1}$, if $u(v+1)\geq 0$ then $u(v) >0$ or $f_v\leq 0$.
    \item For any $v\in\ii{1,n}$, if $u(v-1)\geq 0$, then $u(v)\geq 0$ or $f_{v-1}> 0$.
        \label{i:nonnegative.b}
    \end{enumerate}
    
  \end{lemma}

 \begin{prop}[Proposition 5.16]\thlabel{prop:greedy.path}
   Let
   $(0,0)_0=(r_0,s_0)_0\to(r_1,s_1)_1\to\cdots$ be the
   $k$-greedy infection path. There exist constants $C,c$
   depending only on $\lambda$ and $k$ such that for all $n$ and all $t\geq 5$,
   \begin{align}
     \P\Biggl(\abs[\bigg]{r_n-\frac{\crist[k]n^2}{2}}\geq tn^{3/2}\Biggr) &\leq C\exp\biggl(-\frac{ct^2}{1+\frac{t}{\sqrt{n}}}\biggr)\label{eq:greedy.path.u},\\\intertext{and}
     \P\biggl(\abs[\Big]{s_n-\crist[k]n}\geq t\sqrt{n}\biggr)&\leq 2e^{-ct^2}.\label{eq:greedy.path.s}
   \end{align}
 
 \end{prop}

\begin{lemma}[Special case of Lemma~2.3, the least-action principle]\thlabel{lem:lap}
  Let $u$ be the true odometer stabilizing finite $V\subseteq\ZZ$ with given instructions and initial configuration
  with no sleeping particles.
  Let $u'$ be an odometer on $\ZZ$ that is stable on $V$ for the same instructions and 
  initial configuration.
  Then
  \begin{align*}
    u(v)\leq u'(v)
  \end{align*}
  for all $v\in V$.
\end{lemma}

\bibliographystyle{amsalpha}
\bibliography{main}

\end{document}